\documentclass[10pt]{amsart}
\usepackage{amssymb,amsmath,amsfonts,amsthm}
\usepackage{graphicx}
\usepackage{color}
\usepackage{psfrag}
\usepackage{hyperref}
\usepackage[normalem]{ulem}

\newtheorem{theorem}{Theorem}[section]
\newtheorem{lemma}[theorem]{Lemma}
\newtheorem{proposition}[theorem]{Proposition}

\theoremstyle{definition}
\newtheorem{definition}[theorem]{Definition}

\newtheorem{remark}[theorem]{Remark}

\newtheorem{ltheorem}{Theorem}

\def\real{\mathbb{R}}

\def\rational{\mathbb{Q}}
\def\integer{\mathbb{Z}}
\def\natural{\mathbb{N}}
\def\bS{\mathbb{S}^1}

\def\d{\operatorname{d}}

\def\cA{\mathcal{A}}

\def\nU{\mathcal{U}}
\def\nV{\mathcal{V}}

\def\quand{\quad\text{and}\quad}
\def\hf{\hat{f}}
\def\hA{\hat{A}}
\def\hB{\hat{B}}

\def\hmu{\hat{\mu}}

\def\hF{\hat{F}}

\def\hX{\hat{X}}
\def\hx{\hat{x}}
\def\hy{\hat{y}}

\def\proj{\mathbb{R}\mathbb{P}^{1}}
\def\SL{SL_d(\real)}
\def\SL2{SL_2(\real)}
\def\Diag{D_d(\real)}
\def\GL{GL_d(\real)}

\def\sims{\sim^s}
\def\simu{\sim^u}

\def\Leb{{\operatorname{Leb}}}

\newcommand{\norm}[1]{{\left\lVert  #1  \right\rVert}}
\newcommand{\abs}[1]{{\left\lvert  #1  \right\rvert}}

\title{Random product of quasi-periodic cocycles}
\author{Jamerson Bezerra}
\author{Mauricio Poletti}
\date{\today}

\newcommand{\information}{{
  \bigskip
  \footnotesize

  \medskip
  \textbf{Mauricio Poletti}: \textsc{CNRS-Laboratoire de Math\'ematiques d'Orsay, UMR 8628, Universit\'e Paris-Sud 11, Orsay Cedex 91405, France } \par\nopagebreak
  \textit{E-mail:} \texttt{mpoletti@impa.br}

  \medskip
  \textbf{Jamerson Bezerra}: \textsc{IMPA-Instituto Nacional de Matemática Pura e Aplicada, Rio de Janeiro, CEP: 22460-320, Brazil}
  \par\nopagebreak
  \textit{E-mail:} \texttt{j.douglas.santos@gmail.com}
}}
  
\begin{document}
  
\begin{abstract}
Given a finite set of quasi-periodic cocycles the random product of them is defined as the random composition according to some probability measure.

We prove that the set of $C^r$, $0\leq r \leq \infty$ (or analytic) $k+1$-tuples of quasi periodic cocycles taking values in $\SL2$ such that the random product of them has positive Lyapunov exponent contains a $C^0$ open and $C^r$ dense subset which is formed by $C^0$ continuity point of the Lyapunov exponent

For $k+1$-tuples of quasi periodic cocycles taking values in $\GL$ for $d>2$, we prove that if one of them is diagonal, then there exists a $C^r$ dense set of such $k+1$-tuples which has simples Lyapunov spectrum and are $C^0$ continuity point of the Lyapunov exponent.
\end{abstract}

\maketitle

\section{Introduction}

When studying the Lyapunov exponents of linear cocycles two main questions appears frequently: \emph{ Are the exponents continuous with respect to the cocycle and how frequently do we have the maximum number of different exponents?}

We say that the Lyapunov spectrum of a cocycle is \emph{simple} if it has the maximum number of different exponents (this is the case when all Oseledets spaces are one dimensional). If we deal with two dimensional cocycles, this means our matrices take values in $\SL2$, much more is known than in the higher dimensional case, when it takes values in $\GL$, $d>2$.

It was proved by Bochi ~\cite{Boc02} that in the $C^0$ topology generically two dimensional cocycles either have uniform hyperbolicity or have only one exponent. Another result, by Avila~\cite{Av11}, says that in the smooth topology there is a dense set of $C^r$-cocycles for $0\leq r\leq \infty$ with simple spectrum. In particular these two results imply that in the $C^0$ topology the exponents do not behave continuously, moreover, the only continuity points are the hyperbolic and cocycles with only one exponent. 

This motivates the following question: \emph{Are there open and dense sets of cocycles with simple spectrum, or, in other words, is the simplicity of the spectrum a generic property in smooth topology?}

For two dimensional cocycles over bases with some hyperbolic behaviour this was proved in many scenarios (see \cite{Almost}, \cite{ASV13}, \cite{Pol18}).

For two dimensional smooth quasi-periodic cocycles, ($C^r$ topology, $0<r\leq \infty$) Wang and You proved in \cite{WaY} that the set with simple spectrum is not open, in particular the exponents are not continuous with respect to the cocycle.

For higher dimensional cocycles the problem of simplicity becomes more delicate, this problem goes back to the works of Guivarc'h-Raugi~\cite{GR86} and Gol'dsheid-Margulis \cite{GM89}, where they work with random product of matrices. By \emph{random product of matrices} we mean the cocycle generated by composing randomly a set of matrices accordingly to a probability in the group of matrices. 
They prove simplicity of the Lyapunov spectrum in this scenario with some generic conditions on the support of the probability measure. 

For more general higher dimensional cocycles over hyperbolic maps, if some bunching conditions are assumed, simplicity was proved to be generic in the $C^r$ topology, for $0<r<\infty$ (\cite{BoV04}, \cite{AvV1}, \cite{BPV19}).

For quasi-periodic cocycles very few is known in higher dimension about simplicity, see for example~\cite{DuK} where they find some strong conditions to have simplicity of the biggest Lyapunov exponents.

In this work we study the Lyapunov exponents 
of cocycles over dynamics that have both behaviors, a random part and a isometric (quasi periodic) part. Specifically we deal with random product of quasi-periodic cocycles by this we mean the cocycle generated by taking a probability measure in the set of quasi-periodic cocycles and iterating randomly according to this probability.

We prove that, for two dimensional cocycles, there exists an open and dense set of cocycles that are points of continuity for the Lyapunov exponents and have simple spectrum, even in the $C^0$ topology, and in higher dimension, with one of them taking values in the diagonal group, we find open and dense sets with simple spectrum.

Observe that in contrast to the non-random case, where, in the $C^0$ topology, generically we have one Lyapunov exponent outside of uniform hyperbolicity, in the random product scenario we get generically positive Lyapunov exponents also in the $C^0$ topology. This change of behavior when we ad some randomness on the dynamics was already exploited in many cases, see for example \cite{Via97}, \cite{LSSW03}, \cite{BXY17}.

The precise statements are given in the following section

\section{Definitions and Statements}

Given an invertible measurable map $f:M \rightarrow M$ and a measurable application $A: M \rightarrow \GL$, we define the linear cocycle as the map $F:M\times\real^d\rightarrow M\times\real^d$ given by
\begin{eqnarray*}
 F(x,v) = (f(x),A(x)v).
\end{eqnarray*}
Usually we denote the linear cocycle by the pair $(f,A)$ and, sometimes, when the map $f$ is fixed (and there is no ambiguity) we denote just by $A$. Its iterates are given by $F^n(x,v) = (f^n(x), A^n(x)v)$, where
\begin{eqnarray*}
 A^n(x) = 
\left\{
\begin{array}{ll}
A(f^{n-1}(x))\cdots A(x),     &\mbox{if}\ \ n \geq 1 \\
Id,     &\mbox{if}\ \ n = 0 \\
A(f^{n}(x))^{-1}\cdots A(f^{-1}(x))^{-1}      &\mbox{if}\ \ n \leq 1\\
\end{array}
\right.
\end{eqnarray*}

In the case that $f$ preserves a probability measure $\mu$ which is ergodic and  $\log||A^{\pm 1}||$ is $\mu$-integrable, Oseledets theorem~\cite{Ose68} says that there exists
$k \in \natural$, real numbers $\lambda_1(A)>...>\lambda_k(A)$ and a decomposition of $\real^d = E^1(x)\oplus\cdots\oplus E^k(x)$ by measurable subspaces $E^i(x)$ such that for $\mu$-a.e. $x \in M$ we have
\begin{itemize}
\item $A(x)E^i(x) = E^i(f(x))$ for all $i$, and 
\item $\displaystyle\lim_{|n|\to \infty}||A^n(x)v|| = \lambda_i(A),$ for all $v \in E^i(x)$.
\end{itemize}
  
The numbers $\lambda_1(A),\cdots,\lambda_k(A)$ are called the {\it Lyapunov exponents} associated with the cocycle $(f, A)$ and the set formed by them is called {\it Lyapunov spectrum}, we say that $A$ has \emph{simple spectrum} if $k=d$. 

When $M = \bS$ and $f$ is a rotation of angle $\theta \in (0,1]$ we say that the cocycle $(f,A)$ is a {\it quasi-periodic cocycle} and usually write $(\theta, A)$.

Fix $\theta_i \in (0,1]$ for all $i\in I_k=\{0,\dots,k\}$ and $\nu = \sum_{i=0}^k \nu_i\delta_i$ a probability measure on $I_k$. We will denote by $\Leb$, the Lebesgue measure on $\bS$. Let $X=I_k^{\integer}$, $\hX = X\times \bS$ and consider the invertible (locally constant) skew product 
$\hf:\hX \rightarrow \hX$, given by
\begin{eqnarray*}
\hf((x_n)_{n \in \integer},t) = ((x_{n+1})_{n \in \integer}, t + \theta_{x_0})
\end{eqnarray*}
and observe that the measure $\hmu = \nu^{\integer}\times \Leb$ is $\hf$-invariant. We also write $f_{i}:\bS\to \bS$, $f_{i}(t)=t + \theta_i$, $f^n_x = f_{x_{n-1}}\circ\cdots\circ f_{x_0}$ and $f^{-n}_x = f^{-1}_{x_{-n}}\circ\cdots\circ f^{-1}_{x_1}$, for all $n \in \natural$ and $x \in X$.

From now on we will fix $\theta_0$ irrational. In particular, we have that $(\hf,\hmu)$ is an ergodic system.

For measurable maps $A_i: \bS \rightarrow \GL$, $i \in I_k$, we define the \emph{random product} of the quasi periodic cocycles $(\theta_i,A_i)_{i\in I_k}$ as the cocycle $(\hf, \hA)$, where $\hA: \hX \rightarrow \GL$ is given by $\hA(x, t) = A_{x_0}(t)$.

Note that each vector $(A_0,...,A_k) \in (C^r(\bS,\GL))^{k+1}$ defines a random product $(\hf,\hA)$. We abuse of the notation and denote by $\hA$ the above map and the point $(A_0,...,A_k)$ which define $\hA$.

Let $0\leq r\leq \infty$ (or $r  = \omega$ for the analytic case), by $C^s$ topology in the product space $(C^r(\bS,\GL))^{k+1}$, for $s \in [0,r]$ (or $s \in [0,\infty]\cup\{\omega\}$ in the case when $r = \omega$) we mean the topology given by the distance
$$
\d_{C^s}(\hA,\hB)=\max_{i\in I_k}\d_{C^s(\bS,\GL)}(A_i,B_i),
$$
where $\hA,\hB \in (C^r(\bS,\GL))^{k+1}$ and $\d_{C^s(\bS,\GL)}$ denote the distance that generates the $C^s$ topology in $C^r(\bS,\GL)$.

We say that the $\hA\in (C^r(\bS,\GL))^{k+1}$ is a $C^s$-continuity point for the Lyapunov exponents if for all sequences $\hA_k \in C^r(\bS,\GL)$ corveging to $\hA$ in the $C^s$-topology we have that the Lyapunov exponents of $\hA_k$ converge to the Lyapunov exponents of $\hA$.

These cocycles where already introduced in \cite[section~5.2]{BPS18}, to show that cocycles over some partially hyperbolic dynamics are not continuous in general.

If our cocycle $A$ takes values in $\SL2$ recall that we have $\lambda_1(A)=-\lambda_2(A)$, in particular we have simple spectrum if and only if we have one positive exponent. If the cocycle takes values in the space of two by two matrices with positive determinant, we can reduce to $\SL2$ just dividing by the square root of the determinant.

\begin{ltheorem}\label{teo.positive.exp}
For $r \in [0,\infty]\cap\{\omega\}$, there exists a $C^0$ open and $C^r$ dense subset of $(C^r(\bS,\SL2))^{k+1}$, such that the random product defined by cocycles in this set has positive Lyapunov exponent and is a $C^0$-continuity point for the Lyapunov exponents.
\end{ltheorem}

Let $\varphi: \bS \rightarrow \real$ be a continuous function. The \emph{Schrodinger cocycle} associated to the function $\varphi$ is defined as $(f,A_{\varphi})$, where $f: \bS\rightarrow\bS$ and $A_\varphi: \bS \rightarrow \SL2$ is given by

\begin{eqnarray*}
A_\varphi(x) =
\left(
\begin{array}{cc}
 \varphi(x) & -1 \\
 1          &  0 \\
\end{array}
\right).
\end{eqnarray*}
In the literature it is common to use the function $\varphi(x) = E - u(x)$, with $E \in \real$ and $u:\bS \rightarrow \real$ a continuous function. The reason for that notation is the relation of the Schrodinger cocycles with the Schrodinger operator $H_{u,x}: l^2(\integer)\rightarrow l^2(\integer)$,
$$
\left(H_{u,x}(z)\right)_n = z_{n+1} + z_{n-1} + u(f^n(x))z_n,
$$
given by the eigenvalue equation
$$
H_{u,x}(z) = E\cdot z.
$$
For a detailed survey on this topic see~\cite{Dam17}.

When $f$ is a rotation of angle $\theta \in (0,1]$ in $\bS$, Schrodinger cocycles, $(f,A_{\varphi})$, are quasi periodic cocycles which we will denote just by $(\theta,A_{\varphi})$.

\begin{ltheorem}\label{Schrodinger case}
 For $r \in [0,\infty] \cup \{\omega\}$, there exists a $C^0$ open and $C^r$ dense subset of $(C^r(\bS,\real))^{k+1}$, for $0\leq r\leq \infty$, such that the random product defined by the Schrodinger cocycles associated with the functions in this subset has positive Lyapunov exponent and is a $C^0$-continuity point for the Lyapunov exponents.
\end{ltheorem}

In order to state the result in higher dimensions let $\Diag$ be the subgroup of diagonal matrices in $\GL$.

\begin{ltheorem}\label{Main Theorem 1}
 For $d > 2$ and $r \in [0,\infty]\cup\{\omega\}$, there exists a $C^r$ dense subset of $C^r(\bS,\Diag)\times (C^r(\bS,\GL))^k$ such that the random product defined by cocycles in this set has simple Lyapunov spectrum and is a $C^0$-continuity point of the Lyapunov exponents.
 If $r \in [1,\infty]\cup\{\omega\}$ this set is also $C^1$ open.
\end{ltheorem}

Observe that, since the $C^r$ dense set in the Theorem \ref{Main Theorem 1} is formed by cocycles with simple Lyapunov spectrum and $C^0$-continuity points of the Lyapunov exponents, we have a $C^0$ open and $C^r$ dense set with simple Lyapunov spectrum for each $r \in [0,\infty]\cup\{\omega\}$.

\begin{remark}
We say that two cocycles $A,B\in C(\bS,\GL)$ over $f:\bS \to \bS$, are $C^r$-cohomologous if there exists $C\in C^r(\bS,\GL)$ such that $A(t)=C^{-1}(f(t))B(t)C(t)$. Cohomologous cocycles have the same Lyapunov exponents. As a consequence, Theorem \ref{Main Theorem 1},  is also valid for cocycles such that $A_0$ belongs to the set of cocycles $C^r$-cohomologous to cocycles taking values in $\Diag$.
\end{remark}

\section{Holonomies}\label{s.preliminary}

Given $x\in X$, we define its stable set as
$$
W^s(x)=\{y\in X,\text{ such that for some }k\geq 0,\, y_i=x_i\text{ for }i\geq k \},
$$
and the unstable set as 
$$
W^u(x)=\{y\in X,\text{ such that for some }k\leq 0,\, y_i=x_i\text{ for }i\leq k \}.
$$
We write $x\sims y$ if $x$ and $y$ are in the same stable set and similarly we write $x\simu y$ if $x$ and $y$ are in the same unstable set.

If $x \sims y$ we define the \emph{stable holonomy} from $x$ to $y$, $h^s_{x,y}:\bS\rightarrow \bS$, as 
$$
h^s_{x,y}=\lim_{n\to \infty}(f^n_y)^{-1}\circ f^n_x=(f^{n_0}_y)^{-1}
\circ f^{n_0}_x.
$$
where $n_0$ is the smallest integer such that $x_i = y_i$ for all $i \geq n_0$. Analogously, we define for $x\simu y$ the \emph{unstable holonomy} from $x$ to $y$, $h^u_{x,y}:\bS\rightarrow \bS$, as
$$
h^u_{x,y}=\lim_{n\to -\infty}(f^n_y)^{-1}\circ f^n_x=(f^{n_0}_y)^{-1}
\circ f^{n_0}_x.
$$
where $n_0$ is the biggest integer such that $x_i = y_i$ for all $i \leq n_0$..

Now consider $\hx,\hy\in \hX$, $\hx=(x,t)$ and $\hy=(y,t')$. We write $\hx\sims \hy$ if $x\sims y$ and $t' = h^s_{x,y}(t)$ and we write $\hx\simu\hy$ if $x\simu y$ and $t'=h^u_{x,y}(t)$.

Define the \emph{Linear Stable Holonomies}  associated with the random product $(\hf, \hA)$ as the family of linear maps $\{H^{s, \hA}_{\hx,\hy}: \real^d\rightarrow \real^d; \hx,\hy \in \hX, \hx\sims\hy\}$, given by

$$
H^{s,\hA}_{(x,t)(y,t')}=\lim_{n\to \infty} (\hA^n(y,t'))^{-1} \hA^n(x,t)=(\hA^{n_0}(y,t'))^{-1} \hA^{n_0}(x,t).
$$
where $n_0 \geq 1$, is such that $x_i=y_i$ for $i\geq n_0$. 

Analogously we define the \emph{Linear Unstable Holonomies} $\{H^{u, \hA}_{\hx,\hy}: \real^d\rightarrow \real^d; \hx,\hy \in \hX, \hx\simu\hy\}$.
We use the notations $\sim^*$ or $H^{*,\hA}_{\hx,\hy}$, meaning  that the sentence remains true for any $* \in \{s,u\}$.

Observe that for fixed $\hx,\hy \in \hX$ with $\hx\sim^* \hy$ the map 
$$
\hA \in (C^0(\bS,\GL))^{k+1}\mapsto H^{*,\hA}_{\hx,\hy}$$
varies continuously in the $C^0$ topology.

Now we will fix some notations that will be essentials in the following sections.

From now on, $p \in X$ will be the fix point of the shift map defined by $p_i=0$ for all $i \in \integer$, $z\in X$ will the homoclinic intersection point defined by $z_0 = 1$ and $z_i = 0$ for all $i\neq 0$. Let $t' = h^u_{p,z}(t)$ and define
\begin{equation}\label{eq.hH}
h=h^s_{z,p}\circ h^u_{p,z}\quand H^{\hA}_t=H^{s,\hA}_{(z,t')(p,h(t))} \circ H^{u,\hA}_{(p,t)(z,t')}.
\end{equation}
Observe that for any fixed $t\in \bS$ the map $\hA\mapsto H^{\hA}_t$ varies continuously in the $C^0$ topology. We omit the index $\hA$ when the random product that we are dealing with is clear by the context.

\section{proof of theorem~\ref{teo.positive.exp}}

In this section all cocycles takes values in $\SL2$, in particular, for any map $A$, we have at most $2$ exponents $\lambda_+(A)$ and $\lambda_{-}(A)$.
So, in this section and in the next the term Lyapunov exponents refers to $\lambda_+$ and this is not a restriction since $\lambda_+(A)=-\lambda_-(A)$.

Recall that the system $(\hf,\mu)$ is ergodic, since we are assuming that $\theta_0\in \real\setminus \rational$. Let $(\hf,\hA)$ be the random product of $(A_0,...,A_k) \in (C^0(\bS,\SL2))^{k+1}$.

\begin{definition}[Weakly pinching]
We say that the cocycle $(\hf,\hA)$ is weakly pinching if the cocycle $(\theta_0,A_0)$ has $\lambda_+(A_0)>0$ with respect to the Lebesgue measure on $\bS$.
\end{definition}

By Oseledets theorem, if $\lambda_+(A_0) > 0$, there exists a measurable map $\bS\ni t\mapsto (e_+(t),e_-(t))\in \proj\times \proj$ where $e_+(t)$ is the direction of the Oseledets decomposition corresponding to $\lambda_+(A_0)$ and $e_-(t)$ the direction corresponding to $\lambda_-(A_0)$. 

Take $p$, $z$, $h$ and $H_t$ as defined in the end of Section \ref{s.preliminary}.
\begin{definition}[weakly twisting]
We say that a weakly pinching cocycle $(\hf,\hA)$ is weakly twisting if 
$$
H_t(\{e_+(t),e_-(t)\})\cap \{e_+(h(t)),e_-(h(t))\}=\emptyset
$$ 
for a positive measure subset of $t\in \bS$
\end{definition} 
Observe that both conditions, weakly pinching and weakly twisting, only depends on $(\theta_0,A_0)$ and $(\theta_1,A_1)$.

The pinching and twisting conditions defined here (and in section~\ref{s.proofA}) are generalizations of the ones introduced in \cite{AvV1} for cocycles over hyperbolic dynamics.

We will call the cocycle $(\hf,\hA)$ \emph{weakly simple}, if it is both weakly pinching and weakly twisting.

Define the projective cocycle
$$
P\hF_{\hA}:\hX\times \proj\to \hX\times \proj,\quad (\hx,[v])\mapsto (\hf(\hx),[\hA(\hx)v])
$$
and let $m$ be an $P\hF_{\hA}$-invariant measure that projects on $\mu$.
Using Rokhlin's Disintegration Theorem, we can find a measurable map $\hx\mapsto m_{\hx}$ such that 
$$
m=\int_{\hX} m_{\hx} d\mu\quand m_{\hx}(\{\hx\}\times \proj)=1.$$

We say that $m$ is $u$-invariant if there is a total measure set $X'\subset \hX$ such that for every $\hx,\hy \in X'$ with $\hx\sim^u\hy$, we have
$(H^u_{\hx,\hy})_* m_{\hx}=m_{\hy}$.
Analogously, we say that $m$ is $s$-invariant if the same is true changing unstable by stable holonomies. Finally, $m$ is $su$-invariant if it is both $s$ an $u$ invariant.

By \cite[Proposition~3.9]{ObPol18}, weakly simple implies that the projective cocycle, $PF_{\hA}$, do not admit any $su$-invariant measure.

We have the following criteria.
\begin{proposition}\label{positivity criteria}
Let $r \in [0,\infty]\cup\{\omega\}$ and take $\hA\in (C^r(\bS,\SL2))^{k+1}$. If the random product $(\hf,\hA)$ is weakly simple then, there exists a $C^0$ neighborhood of $\hA$ in $(C^r(\bS,\SL2))^{k+1}$, such that any random product of cocycles in this neighborhood, has postive Lyapunov exponent and is a $C^0$-continuity point of the Lyapunov exponent.
\end{proposition}
\begin{proof}
The invariance principle of \cite{Extremal} (see \cite[Theorem~6.2]{Pol18} for a version that fits into our settings) says that if $\lambda_+(\hA)=0$ then any $m$, $P \hF$ invariant measure, is $su$-invariant. Since, $(\hf,\hA)$ is weakly simple, we conclude that $\lambda_+(\hA)>0$.

Take $\hA_k\to \hA$ such that $\lambda_+(\hA_k)\nrightarrow \lambda_+(\hA)$.
Consider $m_k$, $P \hF_{\hA_k}$ invariant measure that projects to $\hmu$, $u$-invariant such that 
$$
\lambda_+(\hA_k)=\int \log\frac{\norm{\hA_k(\hx)v}}{\norm{v}} d m_k.
$$
Up to taking a sub-sequence we can assume that 
$$m_k \rightharpoonup^* m\quand \lambda_+(\hA_k)\to a<\lambda_+(\hA),$$ where $m$ is an $P\hF_{\hA}$-invariant measure that projects to $\mu$.

As $\lambda_+(\hA)>0$, otherwise it is a continuity point, we have that $m=\alpha m^+ +\beta m^-$, with $\alpha+\beta=1$ where $m^*=\int \delta_{e_*(\hx)}d\mu(\hx)$ and $e_*(\hx)$ is the Oseldets sub-space corresponding to $\lambda_*(\hA)$ for $*=+$ or $-$.
So,  
$$
\int \log\frac{\norm{\hA(\hx)v}}{\norm{v}} d m=\alpha \lambda_+(\hA)+\beta \lambda_-(\hA).
$$
By assumption we have that $\alpha \lambda_+(\hA)+\beta \lambda_-(\hA)<\lambda_+(\hA)$, which implies that $\alpha<1$.

Recall that the holonomies varies continuously in the $C^0$ topology, then for every $\hx\sim^* \hy$ with $*=s\text{ or }u$, we have that $H^{*,\hA_k}_{\hx,\hy}$ converges uniformly to $H^{*,\hA}_{\hx,\hy}$, so we can apply \cite[Theorem~A.1]{Pol18} to conclude that $m$ is $u$-invariant. Observe that the measure $m^-$  is $s$-invariant and $m^-=\frac{1}{\beta}(m-\alpha m^+)$, then $m^-$ is also $u$-invariant. Analogously we conclude that $m^+$ is $su$-invariant.

Since $(\hf,\hA)$ does not admit $su$-invariant measures we conclude that $\hA$ is a continuity point of the Lyapunov exponent $\lambda_+$.

Assume now that there exists a sequence $\hA_k \in (C^r(\bS,\SL2))^{k+1}$, of $C^0$-discontinuity points of the Lyapunov exponent, converging to $\hA$ in the $C^0$ topology. Repeting the above argument we can see that for each $k \in \natural$ we can find a $PF_{\hA_k}$-invariant measure, that projects to $\hmu$ that is $su$-invariant. Passing to a subsequence we can find $m$ such that $m_k$ converges to it, again by \cite[Theorem~A.1]{Pol18} $m$ is an $su$-invariant measure for $P\hF_{\hA}$. This contradicts the fact that the random product $(\hf,\hA)$ does admit any $su$-invariant measure.

Therefore we can find a $C^0$ neighborhood of $\hA \in (C^r(\bS,\SL2))$ such that the random product defined by the cocycle in this neighborhood is a $C^0$ continuity point of the Lyapunov exponent.
\end{proof}

To conclude the proof of theorem~\ref{teo.positive.exp}, we just need to prove that the weakly simple random products are $C^r$ dense.

\begin{proposition}
For $r\in [0,\infty]\cup\{\omega\}$, there exists a $C^r$ dense set of $C^r(\bS,\SL2)^2$ such that for any $(A_0,A_1)$ in this set, the random product of $(\theta_i,A_i)_{i\in I_k}$ is weakly simple for all $(A_2,...,A_k)\in (C^r(\bS,\SL2))^{k-1}$.
\end{proposition}
\begin{proof}
By \cite{Av11} we can find a $C^r$-dense set of $A_0$ such that $\lambda_+(A_0)>0$. In other words we have a $C^r$ dense set weakly pinching.

Now observe that in our setting $H_t=A_0(h(t))^{-1} A_1(t)$.
Take $K\subset \bS$ with $\Leb(K)>\frac{1}{2}$ such that $t\mapsto (e_+(t),e_-(t))$ is continuous in $K$.
Thus, since $h$ preserves the Lebesgue measure we have that $\Leb(K\cap h^{-1}(K))>0$. Take $t\in K\cap h^{-1}(K)$ such that $t$ is a density point for the Lesbegue measure. If
$$
H_t\{e_+(t),e_-(t)\}\cap \{e_+(h(t)),e_-(h(t))\}=\emptyset,
$$ 
then, by continuity, there exists a neighborhood of $t$ in $K\cap h^{-1}(K)$, containing a set with positive measure satisfying the same property which implies that the random product is twisting.

If $H_t\{(e_+(t),e_-(t)\}\cap \{e_+(h(t)),e_-(h(t)\}\neq\emptyset$ we change $A_1$ by $\tilde{A}_1=A_1\circ R_\theta$, for $\theta$ small.

This implies the holonomy of $(A_0,\tilde{A}_1)$ is given by $\tilde{H}_t=H_t R_{\theta}$.
Then we can take $\theta$ arbitrarily small such that 
$$
\tilde{H}_t\{e_+(t),e_-(t)\}\cap \{e_+(h(t)),e_-(h(t))\}=\emptyset.
$$
Consequently, the random product $(\theta_i,A_i)$ is weakly twisting concluding the proof.
\end{proof}

\section{Proof of Theorem \ref{Schrodinger case}}
In the case of Schr\"odinger cocycles the perturbation to get weakly simple cocycles is more delicate because we can only perturb $\varphi_i$, $i\in I_k$. So to conclude the Theorem \ref{Schrodinger case}, using \ref{positivity criteria}, we just need to prove the following Proposition.
\begin{proposition}
For $r \in [0,\infty]\cup\{\omega\}$, there exists a dense subset of $(C^r(\bS, \real))^2$ such that, for any $(\varphi_0,\varphi_1)$ on this subset, the random product of the Schrodinger cocyles $(\theta_i, A_{\varphi_i})_{i \in I_k}$, is weakly simple, for any $(\varphi_2,...,\varphi_k) \in (C^r(\bS, \real))^{k-1}$.
\end{proposition}

\begin{proof}
In \cite{Av11}, it is proved that for any $r \in [0,\infty]$ there is a dense subset of maps $\varphi: \bS\rightarrow \real$ in $C^r(\bS, \real)$ such that the cocycle $(\theta_0,A_{\varphi})$ has positive Lyapunov exponent, which in our case is equivalent to say that for a dense subset of $\varphi_0 \in C^r(\bS,\real)$, the random product of $(\theta_i,A_{\varphi_i})$, with $i = 0,...,k$, is weakly pinching.

Take a point $(\varphi_0,...,\varphi_k) \in (C^r(\bS, \real))^{k+1}$ such that the random product of $(\theta_i, A_{\varphi})_{i \in I_k}$ is weakly pinching and assume, without loss of generality, that $\varphi_0$ does not vanish identically (this can be made since we have density of weakly pinching). Consider $p, z, h$ and $H_t$ as in Section~\ref{s.preliminary} and observe that
\begin{equation}\label{eq.Ht}
H_t = (A_{\varphi_0}(h(t)))^{-1}A_{\varphi_1}(t) = 
\left(
\begin{array}{cc}
 ×1                             &       0\\
 \varphi_0(h(t)) - \varphi_1(t)    &       1\\
 
\end{array}
\right),
\end{equation}
and
$$
h(t) = h_{z,p}^s\circ h_{p,z}^u(t) = t + (\theta_1 - \theta_0)
$$
is a rotation and, in particular, preserves the Lebesgue measure. Observe that the matrix $H_t$ preserve the vertical axis, i.e.  $e_2 = (0,1)$ is the unique fixed point of the action of the matrix $H_t$ in the projective space when $\varphi_1(t)\neq \varphi_0(h(t))$.

With a small change of $\varphi_1$ in the $C^r$ topology we can assume that $\varphi_0(h(t)) \neq \varphi_1(t)$ for $\Leb$-a.e. $t \in \bS$ (for example adding a suitable constant to $\varphi_1$). Hence, $e_2$ is the unique direction in the projective space which is invariant by $H_t$. 

Consider the following set
$$
L_2 = \{t \in \bS; e_2 \in \{e_+(t),e_{-}(t)\}\},
$$
where $e_{+}$ and $e_{-}$ are the Oseledets subspaces associated to the cocycle $(\theta_0,A_{\varphi_0})$ (which we know that has positive Lyapunov exponent).

We claim that $\Leb(\bS\backslash L_2) > 0$. Indeed, otherwise, we have that $e_2$ is a Oseledets subspace of the cocycle $(\theta_0, A_{\varphi_0})$ for $\Leb$-a.e. $t \in \bS$. Then, we have the following possibilities:

\begin{enumerate}
 \item $e_2 = e_+(t) = e_+(f_{\theta_0}(t))$, for some $t\in \bS$:

In this case we have,
$$
e_2 = e_+(f_{\theta_0}(t)) = A_{\varphi_0}(t)e_+(t) = (-1,0),
$$
which is a contradiction.

\item $e_2 = e_+(t) = e_{-}(f_{\theta_0}(t))$, for $\Leb$-a.e. $t \in \bS$:

We have,
$$
e_2 = e_+(f^2_{\theta_0}(t)) = A_{\varphi_0}(f_{\theta_0}(t))A_{\varphi_0}(t)e_+(t) = (-\varphi_0(t), 1),
$$
\end{enumerate}
for $\Leb$-a.e. $t \in \bS$. This shows that $\varphi_0$ vanishes identically, a contradiction.

Therefore, let $\delta = \Leb(\bS\backslash L_2) \in (0,1)$ (the case $\delta=1$ is simpler and follows analogously). By Lusin's theorem, there exists $\Gamma \subset \bS$, such that the functions $e_+$ and $e_{-}$ are simultaneously continuous in $\Gamma$ and $\Leb(\Gamma) > \frac{2 - \delta}{2}$.

Observe that the condition on the measure of $\Gamma$ and the fact that $h$ preserves the Lebesgue measure shows that,
$$
\Leb(\Gamma\cap h^{-1}(\Gamma)) > 1-\delta.
$$
Then,
$$
\Leb(\Gamma\cap h^{-1}(\Gamma)\cap(\bS\backslash L_2)) > 0.
$$

Let $t \in \Gamma\cap h^{-1}(\Gamma)\cap(\bS\backslash L_2)$ be a density point for the Lebesgue measure. Since $e_2 \notin \{e_+(t),e_{-}(t)\}$ we have $e_2 \notin H_t(\{e_+(t),e_{-}(t)\})$. Assume that 
$$
H_t(\{e_+(t),e_{-}(t)\})\cap \{e_+(h(t)),e_{-}(h(t))\} \neq \varnothing.
$$
Then, changing $\varphi_1$ in a small neighborhood of $t$ and using \eqref{eq.Ht} we can make
$$
H_t(\{e_+(t),e_{-}(t)\})\cap \{e_+(h(t)),e_{-}(h(t))\} = \varnothing.
$$
Since $t$ is a density point, $H_t$, $e_+$ and $e_{-}$ are continuous in $t$, we have that the above property is preserved for a positive measure neighborhood of $t$.

So we conclude that, fixed $\varphi_0 \in C^r(\bS,\real)$ such that the cocycle $(\theta_0,A_{\varphi_0})$ has positive Lyapunov exponent, there exists a dense subset of $\varphi_1\in C^r(\bS, \real)$ such that the random product of $(\theta_i, \varphi_i)$ is weakly twisting, for any $\varphi_2,...,\varphi_k \in C^r(\bS, \real)$ and so, weakly simple.
\end{proof}

\section{Proof of the theorem \ref{Main Theorem 1}  }\label{s.proofA}

From now on our cocycles take values on $\GL$, for $d>2$. As always, let $p, z, h$ and $H_t$ be as in section \ref{s.preliminary}.

\begin{definition}[Pinching]
 We say that the random product $(\hf, \hA)$ is \emph{Pinching} if the Lyapunov exponents $\lambda_1(p),\cdots,\lambda_l(p)$ of the cocycle $(f_{0}, A_{0})$  satisfies that $l = d$ and for any $1\leq j\leq d-1$ and $\Leb$-a.e. $t\in \bS$ the sums
\begin{eqnarray}\label{pinching condition}
 \lambda_{i_1}(p)+\cdots+\lambda_{i_j}(p), 
\end{eqnarray}
for all sequences $1\leq i_1<...< i_j \leq d $, are distincts.
\end{definition}

Take $I$ and $J$ subsets of $\{1,\cdots,d\}$ with the same cardinality and consider  the map $P_{I,J}:\GL\rightarrow\real$ defined as the determinant of the matrix obtained taking the minor associated with the lines designated by the elements in $I$ and columns designated by elements of $J$.

\begin{definition}[Twisting]
 We say that the cocycle $(\hf,\hA)$ is \emph{Twisting} if for any $I$ and $J$ as above
 \begin{eqnarray*}
  \int_{\bS}\abs{\log \abs{P_{I,J}(H_t)}}dt < \infty.
 \end{eqnarray*}
 \end{definition}

We say that the random product $(\hf,\hA)$ is \emph{Simple} if it is both Pinching and Twisting. Observe that even if $d=2$ being simple is a stronger condition than being weakly simple.

In order to prove Theorem \ref{Main Theorem 1} we need the following result which is a version of the main result of \cite{PoV18}.
 
\begin{theorem}\label{Continuity condition}
If the cocycle $(\hf,\hA)$ is simple, then the Lyapunov spectrum is simple and it is a continuity point with respect to the $C^0$ topology of the Lyapunov exponents.
\end{theorem}
In \cite{PoV18} the result is stated for H\"older cocycles with some more general dynamics, but as mentioned in \cite[section~4.1]{PoV18} we only need to have well defined holonomies that varies continuously with respect to the cocycle.

Observe $A_0:\bS\to \Diag$ is defined by $d$ functions $a_1,\dots,a_d:\bS\to \real$ such that $(A_0(t))_{i,i}=a_i(t)$ and $(A_0(t))_{i,j}=0$ for $i\neq j$. Then, by Birkhoff's ergodic theorem, the Lyapunov spectrum of $(f_0,A_0)$ is the set
$$\left\{\int \log (a_i)d\Leb
\right\}.$$
So, after suitable choose  of $b_i>0$, we can define $\tilde{a}_i:\bS\to \real$ given by $\tilde{a}_i(t)=b_i a_i(t)$ such that the  diagonal cocycle $\tilde{A}_0$ defined using $\tilde{a}$ has the property \ref{pinching condition}. Moreover, this is a $C^0$ open condition.

Hence after a $C^r$ small perturbation of $A_0$ we can assume that the random product $(\hf, \hA)$ of $(A_0,A_1,...,A_k)$ is always pinching for any $(A_1,...,A_k) \in C^r(\bS,\GL)^k$, and $r\in [0,\infty]\cup\{\omega\}$. 

If $(\hf,\hA)$ is twisting, then, by Theorem \ref{Continuity condition} we have that $(\hf,\hA)$ has simple Lyapunov spectrum and is a $C^0$-continuity point of all Lyapunov exponents. In particular, any $(B_0,...,B_k) \in (C^r(\bS, \GL))^{k+1}$ which is $C^0$-close to $(A_0,...,A_k)$ has also simple Lyapunov spectrum.

So, to conclude the proof of the theorem \ref{Main Theorem 1} it is enough to prove the following theorem:

\begin{theorem}\label{dense twisting}
 Let $d > 2$, $r \in [0,\infty]\cup\{\omega\}$. Then the set of maps $A_1 \in C^r(\bS, \GL)$, such that the random product $(\hf,\hA)$ of $(\theta_i,A_i)_{i\in I_k}$ is twisting, is $C^r$ dense. Moreover for $r \in [1,\infty]\cup\{\omega\}$ this set is also $C^1$ open.
\end{theorem}

\begin{proof} Assume first that $r \in [1,\infty]$. Consider $A_1 \in C^r(\bS, \GL)$ and denote by $(\hf,\hA)$ the random product of $(\theta_i, A_i)$.

Note that, for $p,z \in X$ as in Section \ref{s.preliminary},
$$
h^s_{z,p}(t) = f^{-1}_0\circ f_1\ \ \mbox{and}\ \  h^u_{p,z} = I.
$$
So, $h = h^s_{z,p}\circ h^u_{p,z} = f^{-1}_0\circ f_1.$ Moreover, since $t = h^u_{p,z}(t)$ and observing that

$$
H^s_{(z,t)(p,h(t))} = A_0(h(t))^{-1}A_1(t)\ \ \mbox{and}\ \ H^u_{(p,t)(z,t)} = I,
$$
we get that 
$$
H_t = H^s_{(z,t)(p,h(t))}\circ H^u_{(p,t)(z,t)} = A_0(h(t))^{-1}A_1(t) \in C^r(\bS, \GL).
$$

Therefore, to see that a random product $(\hf, \hA)$ is twisting, we must show that for any $I, J \subset \{1,...,d\}$ with same cardinality the map $P_{I,J}: \GL\rightarrow \real$ satisfies

\begin{eqnarray}\label{main property}
 \int_{\bS}\left|\log \left|P_{I,J}(A_0(h(t))^{-1}A_1(t))\right|\right|dt < \infty.
\end{eqnarray}

The next proposition (which will be proved in Section \ref{s.technical}) is the technical tool in the proof that property \eqref{main property} above holds for $A_0(h(t))^{-1}A_1(t)$ in a dense subset of $C^r(\bS, \GL)$ and we will state it for more general polynomial maps (polynomials with the variable being the coordinates of the matrix in $\GL$), than $P_{I,J}$.

\begin{proposition}\label{Technical tool}
Let $P: \GL \rightarrow \real$ be a non-constant polynomial map, $d \geq 1$ and $r \in [1,\infty]\cup\{\omega\}$. Then, the set
\begin{eqnarray*}
 \cA_P = \left\{ A \in C^r(\bS, \GL); \log\left|P\circ A \right| \in L^1(\bS,m) \right\},
\end{eqnarray*}
is open and dense subset of $C^r(\bS, \GL)$.
\end{proposition}

Using the Proposition \ref{Technical tool} and the continuity of the invertible map $\phi_{A_0}:C^r(\bS, \GL)\rightarrow C^r(\bS, \GL)$ given by
$$
\phi_{A_0}(A)(t) = A_0(h(t))^{-1}A(t),
$$
we can see that the set of $A_1 \in C^r(\bS, \GL)$ such that the random product $(\hf,\hA)$ is twisting, that is,
$$
\left|P_{I,J}\circ (A_0(h(t))^{-1}A(t))\right| \in L^1(\bS, \GL),
$$
for all maps $P_{I,J}$ (there are only finite of them) is open and dense. Hence, we conclude the theorem for $r \in [1,\infty]\cup\{\omega\}$. 

Assume now that $r \in [0,1)$. Then, we can approximate $A_1$ by $B_1 \in C^1(\bS,\GL)$ in the $C^r$ topology and, after that, using the Proposition \ref{Technical tool} again, find $D_1$ close to $B_1$ in the $C^1$ topology (and then close to $A_1$ in the $C^r$ topology) such that the random product of $(\theta_i, D_i)_{i \in I^k}$ is twisting, where $D_j = A_j$ for all $j \neq 1$. This concludes the result for any $r \in [0,\infty]\cup\{\omega\}$.

\end{proof}

\section{Proof of the Proposition \ref{Technical tool}}\label{s.technical}

First we will give a brief review about basic real algebraic geometry. A subset $V\subset\real^l$ is said to be an \emph{Algebraic set} if there exist finitely many polynomials $f_1,...,f_m \in \real[X_1,...,X_l]$ such that

$$
V = \bigcap_{i=1}^m[f_i = 0].
$$

An algebraic set $V$ is said irreducible if, whenever $V = F_1\cup F_2$, with $F_i$ algebraic sets, then $V = F_1$ or $V= F_2$. It is known, see \cite[Theorem~2.8.3,page~50]{BoCoRo2013}, that every algebraic set $V$ can be written as the union of irreducible algebraic sets $V_1,...,V_p$ such that $V_i \nsubseteq\cup_{j\neq i}V_j$. The sets $V_i$ are called the irreducible components of $V$.

For a point $x_0 \in V$, we define the \emph{Zariski tangent space} of $V$ at the point $x_0$ as the linear space
$$
T^{Zar}_{x_0}V = \bigcap_{i=1}^m\left\{x \in \real^l; \nabla f_i(x_0) \cdot x = 0 \right\},
$$
where $\nabla f_i$ denotes the gradient vector of the polynomial $f_i$.

We say that a point $x_0 \in V$ is a \emph{regular} point if
$$
\text{dim}(T^{Zar}_{x_0}V) = \text{min}\left\{ \text{dim}(T^{Zar}_{x}V); x \in V \right\},
$$
and $x_0$ is a \emph{singular} point of $V$ if it is not regular.

If $V_0$ is an irreducible algebraic set we define the dimension of $V_0$ as the number dim$(T^{Zar}_xV)$ for any regular point $x \in V$. For general algebraic sets $V$ we define the dimension as
$$
\text{dim}V = \text{max}\{\text{dim}(V_i); i =1,...,p\},
$$
where $V_i$ are the irreducible components of $V$.

It is important also to reinforce that, when $V$ is an irreducible algebraic set and $x \in V$ is a regular point of $V$, there exists a neighborhood of $x$ in $V$ which is a $C^{\infty}$ manifold and in this case $T^{Zar}_{x}V$ is in fact the tangent space $T_xV$ of this neighborhood at $x$, see \cite[page~66]{BoCoRo2013}.

Let Sing$(V)$ be the set of singular points of $V$ (sometimes we will call this set singular part of $V$ and its complement as regular part of $V$). It is also known, see \cite[Proposition 3.3.14, page~69]{BoCoRo2013}, that Sing$(V)$ is an algebraic subset of $V$ and
\begin{align}
\text{dim}(\text{Sing}(V)) < \text{dim}V.
\end{align}
That will be the crucial property in the proof of the Proposition \ref{Technical tool}.

For more details about algebraic and semi algebraic sets, see \cite{BoCoRo2013}.

Consider $P:\GL\rightarrow \real$ a non-constant polynomial map. Fix $r \in [1,\infty]$ (we will deal with the case $r = \omega$ later) and consider the following algebraic set,
$$
V^0 = [P=0].
$$

Observe that $V^0$ has dimension at most $d^2-1$ (as defined above) and then its regular part is a regular submanifold of $\GL$ of dimension equal to dimension of $V^0$.

Let $V^1 = \mbox{Sing}(V^0) \subset V^0$, be the singular part of $V^0$.

An important tool that we will use is the Thom transversality theorem, that says:

\begin{theorem}[Thom transversality theorem]\label{TTT}
Fix $r \in [1,\infty]$. Let $M$ be a manifold and $N \subset M$ be a submanifold. Then, the set
$$
\{A \in C^r(\bS; M);\ A\pitchfork N\},
$$
is dense. If $N$ is closed then the above set is also open.
\end{theorem}
\begin{remark}\label{disjoint case}
A particular case of this theorem that we will use is when codimension of $N$ is large than $1$. In this case, the unique way to a map $A \in C^r(\bS;M)$ to be transversal to $N$ is  
\begin{align*}
A(\bS)\cap N = \varnothing.
\end{align*}

\end{remark}
For a complete proof and details around this result, see \cite[Corollary~4.12, page~56]{GoGu1973}.

As a corollary of Theorem \ref{TTT} in our context, we have 
\begin{lemma}\label{TTT.2}
Take $A \in C^r(\bS;\GL)$ and assume that $d(A(\bS), V^1) > 0$. Then, for all $\nU \subset C^r(\bS;\GL)$, neighborhood of $A$, there exists $\nV \subset \nU$, open, and $a>0$ such that for all $B \in \nV$ we have:
\begin{itemize}
  \item $d(B(\bS), V^1) > a > 0$
  \item $B \pitchfork V^0.$
 \end{itemize}
\end{lemma}
Observe that as a corollary of this lemma and compacity of $\bS$, we have 
$$
\#B(\bS)\cap V^0 < \infty.
$$
In fact, we can assume that the intersection above has the same cardinality for every $B \in \nV$.

Let $A$ and $\nV$ as in Lemma \ref{TTT.2} and consider $B \in \nV$. Set $\{t_1,\ldots, t_k\}\subset \bS$ such that $\{B(t_1),\ldots, B(t_k)\} = B(\bS)\cap V^0$. By transversality we have that,
$$
(P\circ B)'(t_i) = \nabla P(B(t_i))\cdot B'(t_i) \neq 0, \forall i = 1,\ldots, k. 
$$
Set $s_1 = \frac{1}{2}\min\{|t_i-t_l|; i\neq l\}>0$. By Taylor's formula we have
$$
P\circ B(t) = (P\circ B)'(t_i)(t-t_i) + o_i(|t-t_i|).
$$
Consider $s_2>0$ such that if $|t-t_i|<s_2$ for some $i=1,\ldots,k$, then
$$
\left|(P\circ B)'(t_i) - \frac{o_i(|t-t_i|)}{|t-t_i|}\right| \geq \frac{1}{2}\left|(P\circ B)'(t_i)\right| \neq 0.
$$
In particular, there exists $C_1 > 0$ such that if $|t-t_i|< s_2$, for some $i =1,\ldots,k$, we have
$$
|\log\left|(P\circ B)'(t_i) - \frac{o_i(|t-t_i|)}{|t-t_i|}\right|| \leq C_1.
$$

Define $s = \frac{1}{2}\min\{s_1,s_2\}$ and consider the intervals $I_i = (t_i-s,t_i+s)\subset \bS$. So,
\begin{align*}
\displaystyle\int_{\bS}|\log|P\circ B(t)||dt
&=\displaystyle\int_{\bS\backslash\displaystyle\cup_{i=1}^{k}I_i}|\log|P\circ B(t)||dt + \displaystyle\int_{\displaystyle\cup_{i=1}^k I_i}|\log|P\circ B(t)||dt \\
&= (I) + (II).
\end{align*}
By compacity of $\bS\backslash\displaystyle\cup_{i=1}^k I_i$, and the fact that the function $P \circ B$ does not vanish in this set, there exists another constant $C_2>0$ such that
$$
|\log|(P\circ B(t))|| \leq C_2,
$$
for every $t \in \bS\backslash\displaystyle\cup_{i=1}^k I_i$. This gives
$$
(I)
= \displaystyle\int_{\bS\backslash\displaystyle\cup_{i=1}^{k}I_i}|\log|P\circ B(t)||dt
\leq C_2m(\bS\backslash\displaystyle\cup_{i=1}^k I_k)
$$
Note that by, Taylor's Formula in each $I_i$, we have
\begin{align*}
 \log|P\circ B(t)|
 = \log\left|(P\circ B)'(t_i) + \frac{o_i(|t-t_i|)}{|t-t_i|}\right| + \log|t -t_i|.
\end{align*}
Hence, by the choice of $s$, and the fact that the logarithm funtion is integrable at the origin, we conclude that
\begin{align*}
(II) 
&= \displaystyle\int_{\displaystyle\cup_{i=1}^k I_k}|\log|P\circ B(t)||dt
= \displaystyle\sum_{i=1}^{k}\displaystyle\int_{I_i}|\log|P\circ B(t)||dt\\
&\leq \displaystyle\sum_{i=1}^k\displaystyle\int_{I_i}|\log\left|(P\circ B)'(t_i) + \frac{o_i(|t-t_i|)}{|t-t_i|}\right||dt + \displaystyle\sum_{i=1}^k\displaystyle\int_{I_i}|\log|t-t_i||dt\\
&\leq \displaystyle\sum_{i=1}^k m(I_i)C_1 + C < \infty.
\end{align*}
Therefore, for each $B \in \nV$ we have that $\log|P\circ B| \in L^1(\bS;m)$.

To conclude the proof of the main lemma what is left to show is that the following set
\begin{align}\label{set A}
\cA = \{A \in C^r(\bS;\GL;\ d(A(\bS), V^1) > 0 \},
\end{align}
is dense in $C^r(\bS;\GL)$. This is a consequence, as we will see, of the fact that the $V^1$ is an algebraic variety of codimension greater or equal than one in $\GL$.

Consider the following chain of algebraic sets:
$$
V^m\varsubsetneq V^{m-1}\varsubsetneq \ldots \varsubsetneq V^1 \varsubsetneq V^0,
$$
where $V^i = \mbox{Sing}(V^{i-1})$. We know that this is a finite chain since $V^i$ is an algebraic variety of dimension strictly less than the dimension of $V^{i-1}$. We can assume that $V^m$ is a regular submanifold of $Gl_d(\real)$ (the chain stops in this moment). Moreover, by definition, we have that each $V^i$ is a closed subset of $V^{i-1}$ and $V^{i-1}\backslash V^i$ is a regular submanifold of $\GL$ of codimension large than $1$.

Using the remark \ref{disjoint case} with $M = \GL$ and $N = V^m$ which is a regular submanifold of $\GL$, we have that the set
$$
\cA_m = \{A\in C^r(\bS;\GL);\ A(\bS)\cap V^m = \varnothing\},
$$
is dense in $C^r(\bS;\GL)$. More generally, using the same argument, now with $M = \GL$ and $N = V^i\backslash V^{i+1}$, we will obtain that the set
$$
\cA_i = \{A\in C^r(\bS;\GL);\ A(\bS)\cap V^{i}\backslash V^{i+1} = \varnothing\},
$$
is dense in $C^r(\bS;\GL)$, for every $i = 1,\ldots,m-1$ (note that we are using strongly that the codimension is large enough). Hence, the set
$$
\cA = \displaystyle\bigcap_{i=1}^m\cA_i,
$$
is (open) and dense in $C^r(\bS;\GL)$. Observing that this intersection is the set $\cA$ defined above in \ref{set A}, we conclude the proof of the main lemma in the case $r \in [1,\infty]$.

Assume that $r = \omega$. In this case we have that $P\circ A:\bS\rightarrow \real$ (P is a polynomial and in particular a analytic function) is an analytic function for any $A\in C^{\omega}(\bS, \GL)$. In this case, either $P \circ A$ is constant equal to zero or, has only finitely many zeros and those zeros have finite order in the sense that if we consider a neighborhood of a zero $t_i \in\bS$ and write
$$
P\circ A(t) = (t-t_i)^{m_i}g_i(t),
$$
where $m_i \in \mathbb{N}$, $g_i \neq 0$ on this neighborhood. Since the function $\log t$ is $m$-integrable on $[0,1]$, we conclude that
$$
\log\left|P\circ A\right|\in L^1(\bS,m),
$$
for all $A \in C^{\omega}(\bS, \GL)$ such that $P \circ A$ is not zero. Assume then that $P \circ A \equiv 0$. Since, the polynomial $P$ is non constant, the interior of $[P=0]$ is empty. So, for any $A \in [P=0]$ and for any $\varepsilon > 0$, there exist $B \in \GL$, $||B|| = 1$ and a positive $\delta < \epsilon$ such that $A+\delta B \notin [P =0]$. In particular, taking $A = A (t_0)$, for some $t_0 \in \bS$ and defining $\tilde{A} \in C^{\omega}(\bS,\GL)$ given by
$$
\tilde{A}(t) = A(t) + \delta B,
$$
we have that $\tilde{A}$ is $\delta$-$C^{\omega}$ close to $A$ and $P\circ \tilde{A}$ does not vanish identically. Hence, we fall in the previous case and
$$
\log\left|P\circ\tilde{A}\right|\in L^1(\bS,m).
$$
This concludes the proof of the Proposition \ref{Technical tool}.

\subsection*{Acknowledgements}
The authors will like to thank Lucas Backes and Karina Marin for suggestions that improve the writing of the manuscript. Work partially supported by Fondation Louis D-Institut de France (project coordinated by M. Viana). J.B. was supported by CAPES and M.P by the ERC project 692925 NUHGD.

\bibliography{bib}
\bibliographystyle{plain}

\information

\end{document}